\documentclass[english]{article}
\usepackage[T1]{fontenc}
\usepackage[latin9]{inputenc}
\usepackage{babel}
\usepackage{amsmath}
\usepackage{amsfonts}
\usepackage{amssymb}
\usepackage{amsthm}
\usepackage{caption}
\usepackage{graphicx}
\usepackage{subcaption}
\DeclareGraphicsExtensions{.pdf,.png,.jpg}
\graphicspath{./figure/}
\date{}

\newcommand{\R}{\mathbb{R}}

\newcommand{\Z}{\mathbb{Z}}

\newtheorem{prop}{Proposition}[subsection]

\newtheorem{cor}{Corollary}[subsection]
\newtheorem{lemma}{Lemma}[subsection]
\newtheorem{theorem}{Theorem}[subsection]
\newtheorem*{theorem*}{Theorem}

\theoremstyle{remark}
\newtheorem{remark}{Remark}[subsection]

\theoremstyle{definition}
\newtheorem{defin}{Definition}[subsection]

\begin{document}

\title{The Wriggle Polynomial for Virtual Tangles}

\author{Nicolas Petit}

\maketitle
\begin{center}Oxford College of Emory University (petitnicola@gmail.com)\end{center}

\begin{abstract}
\noindent 
We generalize the Wriggle polynomial, first introduced by L. Folwaczny and L. Kauffman, to the case of virtual tangles. This generalization naturally arises when considering the self-crossings of the tangle. We prove that the generalization (and, by corollary, the original polynomial) are Vassiliev invariants of order one for virtual knots, and study some simple properties related to the connected sum of tangles.
\end{abstract}

\section{Introduction}

The aim of this note is to present a generalization of the Wriggle polynomial for virtual knots, introduced in \cite{linkingnumberaffineindex}, to the case of oriented virtual tangles (of the most general type). The interest in this topic came from the author's previous paper \cite{indexpolyvirtualtangles}, in search of an index-type polynomial invariant for virtual tangles that would behave nicely under connected sum. While we don't think the invariant we found fulfills our wish, we believe it to be a stepping stone in the right direction, because of its connections with the Affine Index polynomial.

This paper is organized as follows: section \ref{virtualknotsandtangles} provides a brief review of the definitions of virtual tangles, linking number, and Vassiliev invariants, that will be needed in the rest of the paper. In section \ref{tangleinvariants} we review the construction of \cite{linkingnumberaffineindex}, then define the polynomial invariants for tangles. We prove our main proposition in section \ref{proof}, and conclude by studying some of the properties of the invariant in section \ref{strengthofinvariants}.

\section{Background}
\label{virtualknotsandtangles}
\subsection{Virtual knots and tangles}

Virtual knots were first introduced by Kauffman in \cite{virtualknottheory}. They can be defined in multiple ways: we can think of a virtual knot as a knot projection (that is, an immersion of $S^1$ into the plane) where every double point is equipped with one of three types of crossing (classical positive, classical negative and virtual), up to Reidemeister moves and virtual Reidemeister moves. We can also think of virtual knots as equivalence classes of Gauss diagrams up to Reidemeister moves, or as equivalence classes of knots in thickened surfaces up to stabilization/destabilization. We will mostly focus on the diagrammatic approach in this paper.

\begin{figure}[!h]
\centering
\includegraphics[scale=0.1]{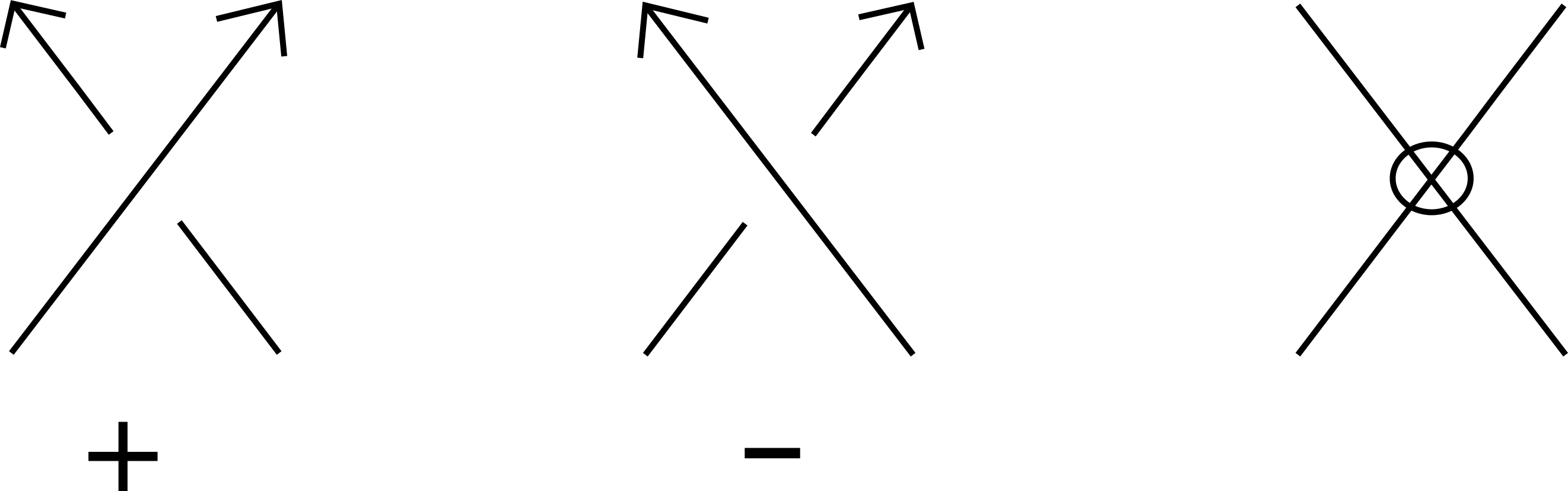}
\caption{The three types of crossing: positive, negative and virtual.}
\end{figure}

\begin{figure}[!h]
\centering
\includegraphics[scale=.07]{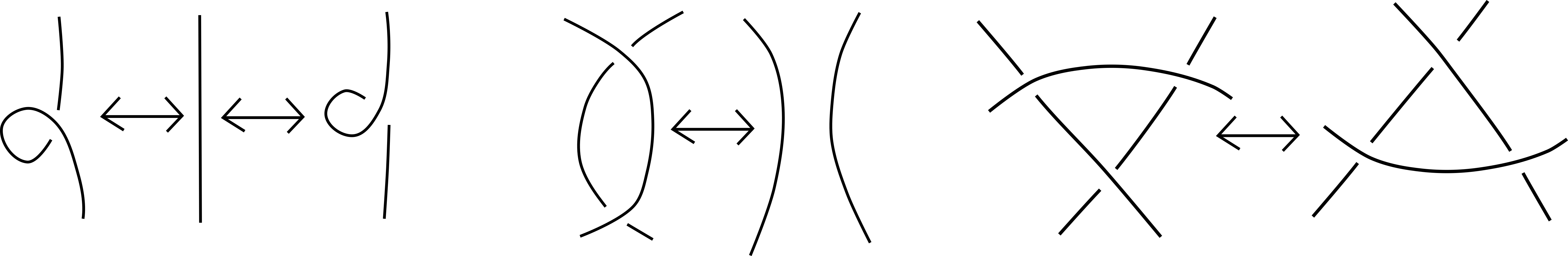}
\caption{The classical Reidemeister moves.}
\label{RM}
\end{figure}

\begin{figure}[!h]
\centering
\includegraphics[scale=.06]{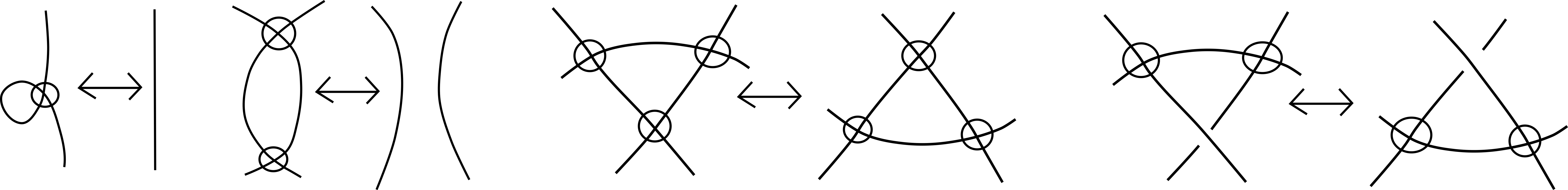}
\caption{The virtual Reidemeister moves.}
\label{virtualRM}
\end{figure}

A long virtual knot differs from a virtual knot in that we're considering immersions of $\R$ into the plane that are constant outside of a ball of finite radius (where all the ``knottiness'' is contained). We assign to every double point one of the three types of crossing, and consider this long virtual knot diagram up to Reidemeister and virtual Reidemeister moves.
These virtual knots also possess the other two interpretations: they are Gauss diagrams with a distinguished point, representing the ``point at infinity'', modulo Reidemeister moves; and can be interpreted as knots in thickened surfaces with boundary, up to stabilization/destabilization. For the way these approaches to long knots are connected, see \cite{chrismanmonoidlongvirtual}, \cite{longframedfti}.

A virtual tangle is a collection of virtual knots and long virtual knots, usually linked with each other in some fashion. 
We will typically represent tangles in a square, where the ends of the long component are fixed distinguished points on the top or bottom sides of the square. 
If the tangle has $m$ endpoints at the top and $n$ endpoints at the bottom, we say that the tangle is an $(m,n)$ tangle. 
We usually draw the top and bottom of the square, but not the sides, see Fig. \ref{virtualtangleexample}.

We can define the connected sum of two virtual tangles as the ``stacking'' of one tangle above another. 
Of course, this operation is only defined if the two tangles have the same number of distinguished points alongside the common boundary, and their orientations match; because of this, the operation is, generally speaking, not commutative. 
We will denote the connected sum of tangle $T$ and $U$ by $T\#U$, in which we stack $T$ above $U$.

\begin{remark}
When talking about a ``tangle'' we will typically work in the most general setting: we allow our tangle to contain closed components, and we don't assume that it has the same number of distinguished points in both boundaries, or that long components necessarily connect top to bottom.
\end{remark}

\begin{figure}[!h]
\centering
\includegraphics[scale=.13]{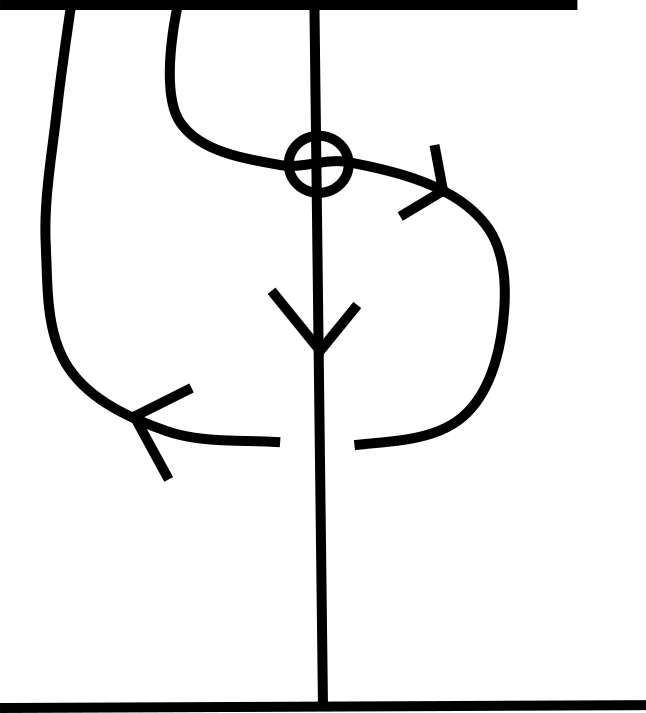}
\caption{An example of a $(3,1)$ virtual tangle.}
\label{virtualtangleexample}
\end{figure}

Finally, let's briefly review the notion of linking number of a virtual link.
For oriented classical links, we have three equivalent definitions of the linking number between two components: we can either take the writhe of all the crossings where component one goes over component two, or take the writhe of all crossings where component two goes over component one, or take the overall writhe of the two components and divide by two.
\begin{equation}\label{linkingnumbereq}lk(T_1, T_2)=\frac{1}{2}\sum_{d\in T_1\cap T_2}sgn(d)=\sum_{d\in T_1\text{ over } T_2}sgn(d)=\sum_{d\in T_2\text{ over } T_1}sgn(d)\end{equation}
As a result of the definition, we clearly have that $lk(T_1, T_2)=lk(T_2, T_1)$, so the linking number only depends on the pair of components.

However, that fact is not true anymore for virtual links, as the number of times $T_1$ goes over $T_2$ need not be the same as the number of times $T_2$ goes over $T_1$. To see this, just take any classical link and change one crossing from classical to virtual; this operation will change one of the last two expressions of equation \ref{linkingnumbereq} but not the other.
To be able to talk about a virtual linking number we must then pick one of the definitions. Throughout this paper, we will follow the literature and use the ``over'' definition, for which $$vlk(T_1, T_2)=\sum_{d\in T_1\text{ over } T_2}sgn(d).$$

Note that the above observations imply that $vlk(T_1, T_2)\neq vlk(T_2, T_1)$, so we need to consider the \emph{ordered} pair of components $(T_1, T_2)$ when computing a virtual linking number. The difference $vlk(T_1, T_2)-vlk(T_2, T_1)$ is called the wriggle number $W(T_1, T_2)$ of the pair of components \cite{linkingnumberaffineindex}; clearly the wriggle number is an invariant of the ordered pair, and $W(T_2, T_1)=-W(T_1, T_2)$, so $|W(T_1, T_2)|$ is an invariant of the (unordered) pair of components.


\subsection{Vassiliev invariants of virtual knots}
\label{vassilievinvariants}
Vassiliev invariants for virtual knots were first introduced by Kauffman in \cite{virtualknottheory} as a natural generalization of the notion of a finite-type invariant for classical knots.
Formally, a Vassiliev invariant of virtual knots is the extension of a virtual knot invariant to the category of singular virtual knots.
These are virtual knots with an extra type of crossing (transverse double points), modulo the Reidemeister moves for virtual knots and some extra moves, called rigid vertex isotopy.

We extend the virtual knot invariant $\nu$ to singular virtual knots by resolving every double point as a weighted average of its two possible resolution, a positive and a negative crossing, see Fig. \ref{doublepointresolution}.
We then say that $\nu$ is a finite-type invariant, or Vassiliev invariant, of order $\leq n$ if it vanishes on every knot with more than $n$ double points. 
Examples of finite-type invariants according to this definition are the coefficient of $x^n$ in the Conway polynomial or the Birman coefficients. 
\begin{figure}[!h]
\centering
\includegraphics[scale=.07]{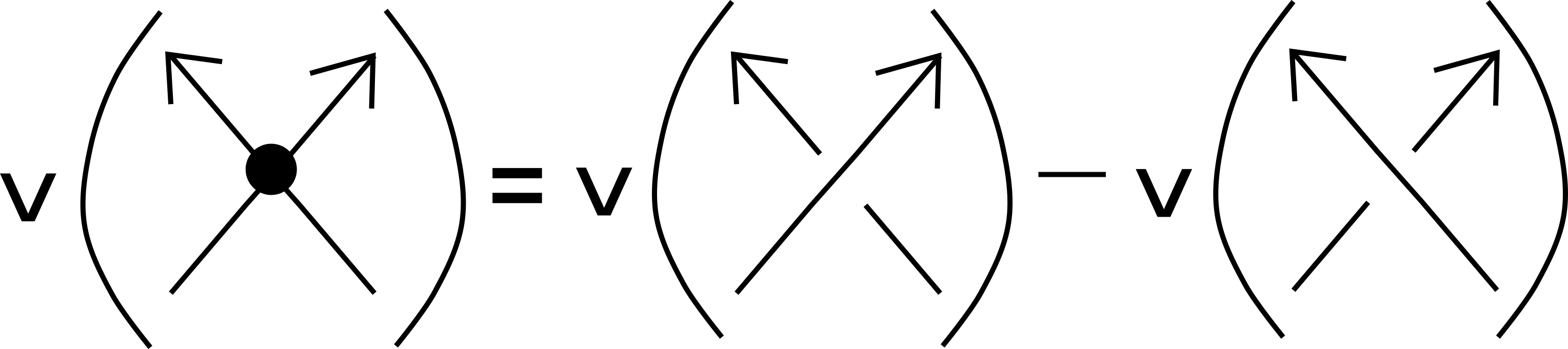}
\caption{How to resolve a double point.}
\label{doublepointresolution}
\end{figure}

We can naturally extend the above ideas to the case of virtual tangles. A finite-type or Vassiliev invariant of virtual tangles will be an extension of a virtual tangle invariant to virtual tangles with double points, where we resolve a double point with the same relation pictured in Fig. \ref{doublepointresolution}. 
Then $\nu$ is a Vassiliev invariant of order $\leq n$ for virtual tangles if it vanishes on any tangle with more than $n$ double points; we say $\nu$ is of order $n$ if it is of order $\leq n$ but not of order $\leq n-1$.

\begin{remark}
Around the same time, Goussarov, Polyak and Viro \cite{GPV} were independently developing a different generalization to virtual knots of finite-type invariants, in terms of a new type of crossing called semi-virtual. We will not work with these finite-type invariants in this paper.
\end{remark}

\begin{remark}
\label{smoothinggivestwocomponents}
We typically smooth crossings according to orientation and give the resulting object the inherited orientation, see Fig. \ref{smoothingofacrossing}.
If both stands belong to the same component, as a result of the smoothing we will get a two-component object.
\end{remark}

\begin{figure}[!h]
\centering
\includegraphics[scale=.1]{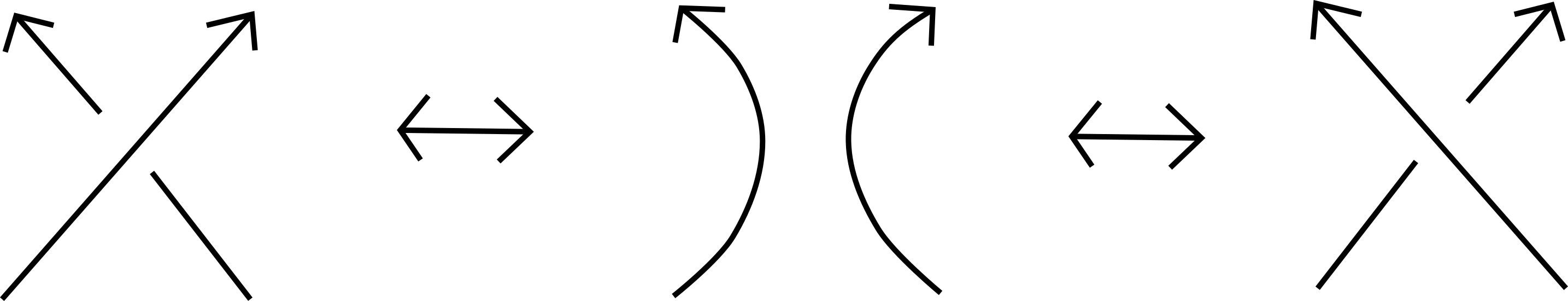}
\caption{How to smooth a crossing according to orientation. Note that positive and negative crossings give the same oriented smoothing.}
\label{smoothingofacrossing}
\end{figure}


\section{The self-crossing wriggle polynomial for virtual tangles}
\label{wrigglepolynomial}
\subsection{Definition of the invariant}
\label{tangleinvariants}
We will briefly recall the definition of the Wriggle polynomial from \cite{linkingnumberaffineindex}, to then extend this invariant to the virtual tangle case.

\begin{defin}
Let $K$ be a virtual knot. At every classical crossing $c$ of $K$, denote the incoming understrand with a black dot, see Figure \ref{blackdot}.

\begin{figure}[!h]
\centering
\includegraphics[scale=.12]{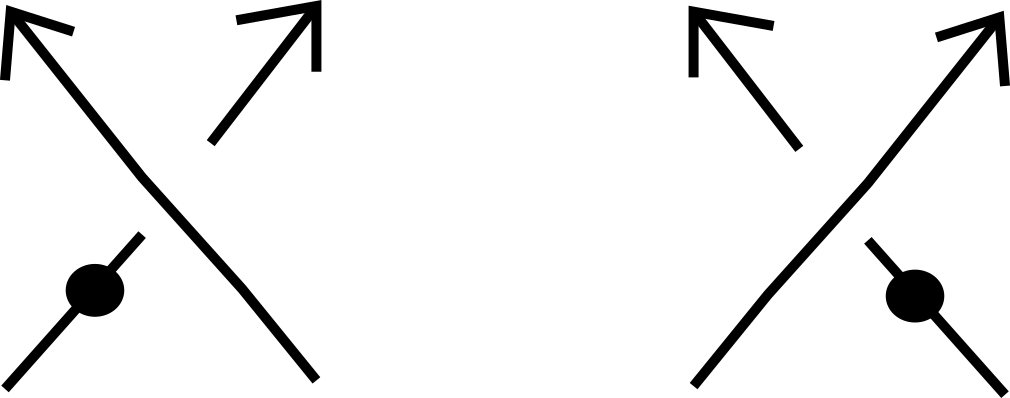}
\caption{How to assign a black dot to every classical crossing of $K$.}
\label{blackdot}
\end{figure}

Smooth each crossing $c$ according to orientation, and compute the wriggle number of the resulting two-component link $L_c$ (see Remark \ref{smoothinggivestwocomponents}), where the component that inherits the dot is labeled to be component one. We define the Wriggle polynomial as
$$W_K(t)=\sum_{c}sgn(c)(t^{W(L_c)}-1)=\sum_{c}t^{W(L_c)} - writhe(K).$$
\end{defin}

Figure \ref{wrigglepolyexample} shows an example of computing the wriggle polynomial for the virtualized trefoil. Smoothing the left crossing yields a link where component $1$ has a positive crossing going under component $2$, so $W(L_c)=0-(1)=-1$; smoothing the right crossings gives a link where component $1$ has a positive crossing going over component $2$, so $W(L_c)=1-0=1$. Both crossings are positive, so the Wriggle polynomial of the knot is $$W_K(t)=+t^1+t^{-1}-2=t+t^{-1}-2$$

\begin{figure}[!h]
\centering
\includegraphics[scale=.12]{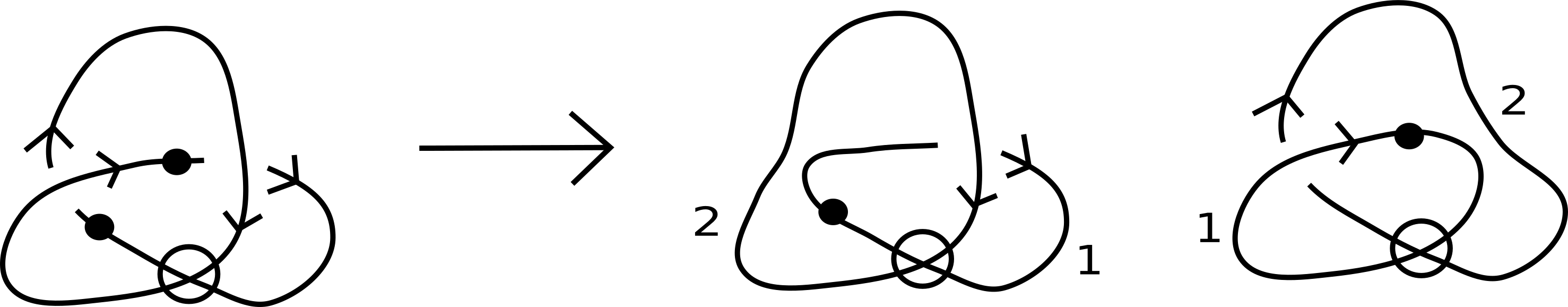}
\caption{Part of a sample computation of the wriggle polynomial.}
\label{wrigglepolyexample}
\end{figure}

\begin{prop}[\cite{linkingnumberaffineindex}] The Wriggle polynomial is a virtual knot invariant. It coincides with the affine index polynomial of \cite{affineindexpolynomial}. \end{prop}

\begin{remark}
Before talking about the general virtual tangles case, we should observe that the wriggle polynomial of a long virtual knot is the same as the wriggle polynomial of its closure. This is easy to see if we look at the Gauss diagram version of the wriggle polynomial, which is explained in detail in \cite{linkingnumberaffineindex} while showing that the wriggle polynomial coincides with the affine index polynomial. We can rewrite the wriggle number as the difference of the number of chords that intersect $c$ positively and the ones that intersect $c$ negatively; since the presence of the point at infinity does not affect either of these numbers, the wriggle polynomial of the long knot will be the same as the one of its closure. As a consequence, if two long knots have the same closure they can't be distinguished by the wriggle polynomial.
\end{remark}

\begin{defin} Let $T$ be an $n$-component oriented virtual tangle, where the components have been ordered $T_1, \ldots, T_n$, and assign to each component a variable $t_i$. Let $c$ be a self-crossing of the component $T_i$; mark the incoming understand of this self-crossing with a black dot. Smooth the self-crossing $c$ according to orientation, and compute the wriggle number $W(L_c)$ of the two-component link obtained by smoothing $c$, \emph{ignoring every other component of T and any crossings they might have with $L_c$}.

The self-crossing wriggle polynomial is then defined as
$$W_{sc, T}(t_1,\ldots, t_n)=\sum_{i}\sum_{c\in T_i}sgn(c) (t_i^{W(L_c)}-1),$$ where the second sum is over all self-crossings of component $T_i$.
\end{defin}

\begin{theorem}\label{wrigglepolyvirtualtanglesinvariant}
The wriggle polynomial for virtual tangles is an order one Vassiliev invariant of virtual tangles.
\end{theorem}

\subsection{Proof of Theorem \ref{wrigglepolyvirtualtanglesinvariant}}
\label{proof}
To successfully prove the theorem, we need to show three things: the polynomial is invariant under Reidemeister moves, its extension vanishes on any virtual tangle with two double points, and there is a virtual tangle with one double point on which the extension is nonzero. Let us first prove that the polynomial is, in fact, a virtual tangle invariant, by checking its invariance under Reidemeister moves.

\begin{proof}
We should note that, since the invariant only cares about self-crossings, we only need to consider the strand configurations in Reidemeister moves that form at least one self-crossing. In Fig. \ref{reidemeisterone} we can easily see that, no matter which component inherits the dot, the wriggle number of the two resulting components is zero, so the contribution of the kink is $\pm(t^0-1)=0$, which is the same as the contribution of the strand without the kink; this shows invariance under the first Reidemeister move. 

\begin{figure}[!h]
\centering
\includegraphics[scale=.1]{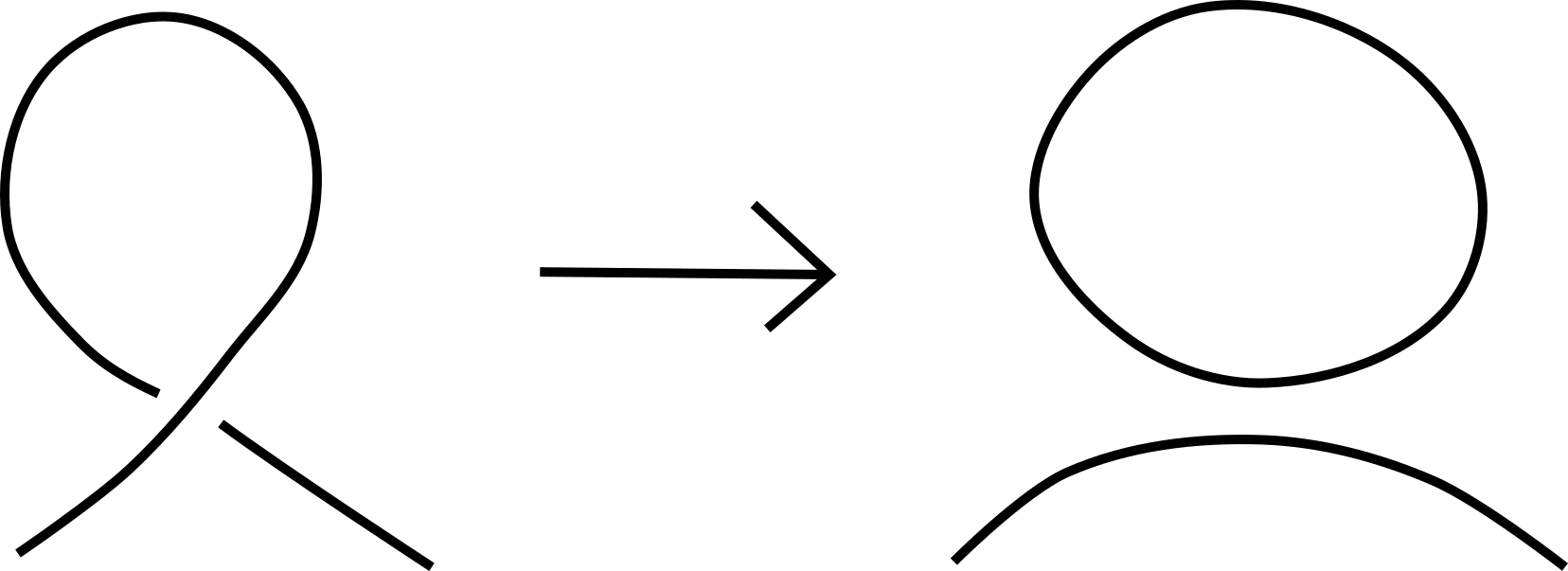}
\caption{The first Reidemeister move. Regardless of which component inherits the dot, the two components we get from the smoothing don't intersect, so $W(L_c)=0$.}
\label{reidemeisterone}
\end{figure}

The cases of the second Reidemeister move are covered in Fig. \ref{reidemeistertwo}: the links that result after the smoothing are isotopic (up to Reidemeister move one), and the dot ends up in the same component, so their wriggle number is identical. Since one crossing is positive and the other is negative, the two contributions cancel out, which shows that the invariant doesn't change under Reidemeister move two.

\begin{figure}[!h]
\begin{subfigure}{.5\textwidth}
  \centering
  \includegraphics[width=.8\linewidth]{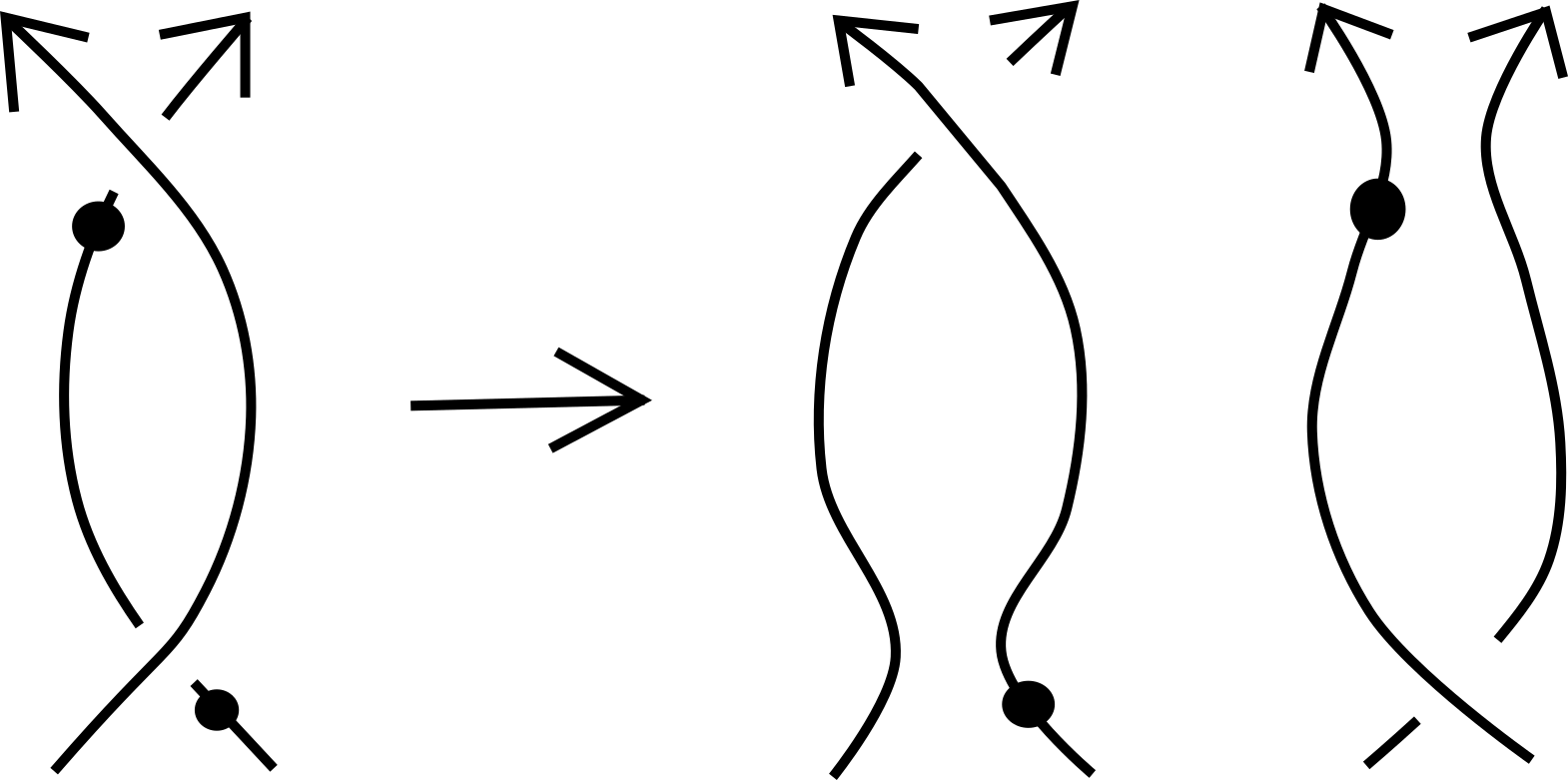}
  \caption{The braided version of R2}
\end{subfigure}%
\begin{subfigure}{.5\textwidth}
  \centering
  \includegraphics[width=.8\linewidth]{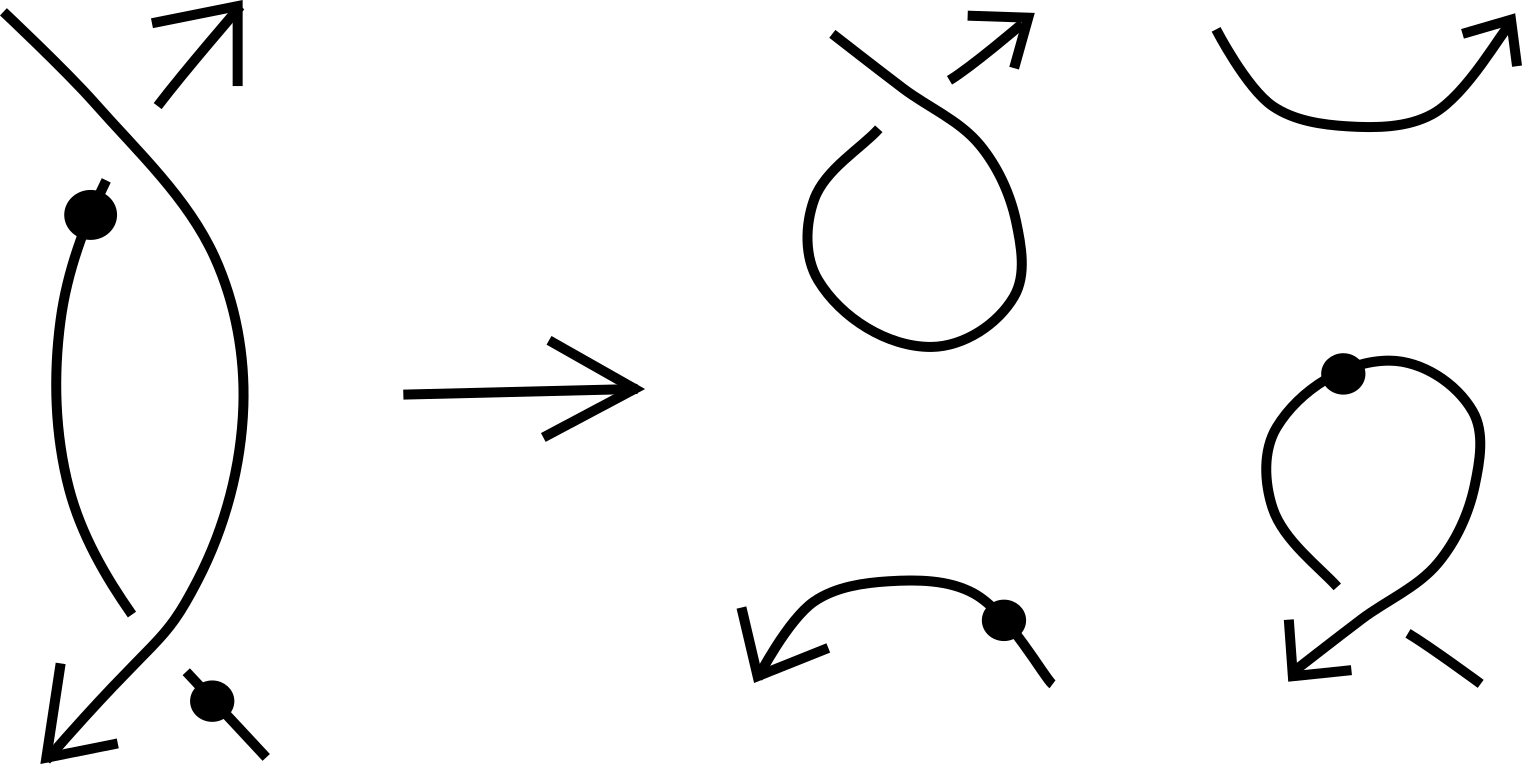}
  \caption{The cyclic version of R2.}
\end{subfigure}
\caption{The two cases of Reidemeister move two.}
\label{reidemeistertwo}
\end{figure}

Let's look at Reidemeister move three: there is a correspondence between the crossings before and after the move, as pictured by Fig. \ref{reidemeisterthree}. If all three strands belong to the same component of the tangle we'll need to consider the contributions of all crossings, otherwise the two strands that belong to the same component will only intersect in one of the crossings. For two of these crossings the smoothing yields homotopic links with the same choice of component one, as pictured in Fig. \ref{reidemeisterthreea} and \ref{reidemeisterthreeb}; these crossings will then contribute the same amount before and after the move. 

\begin{figure}[!h]
  \centering
  \includegraphics[scale=.12]{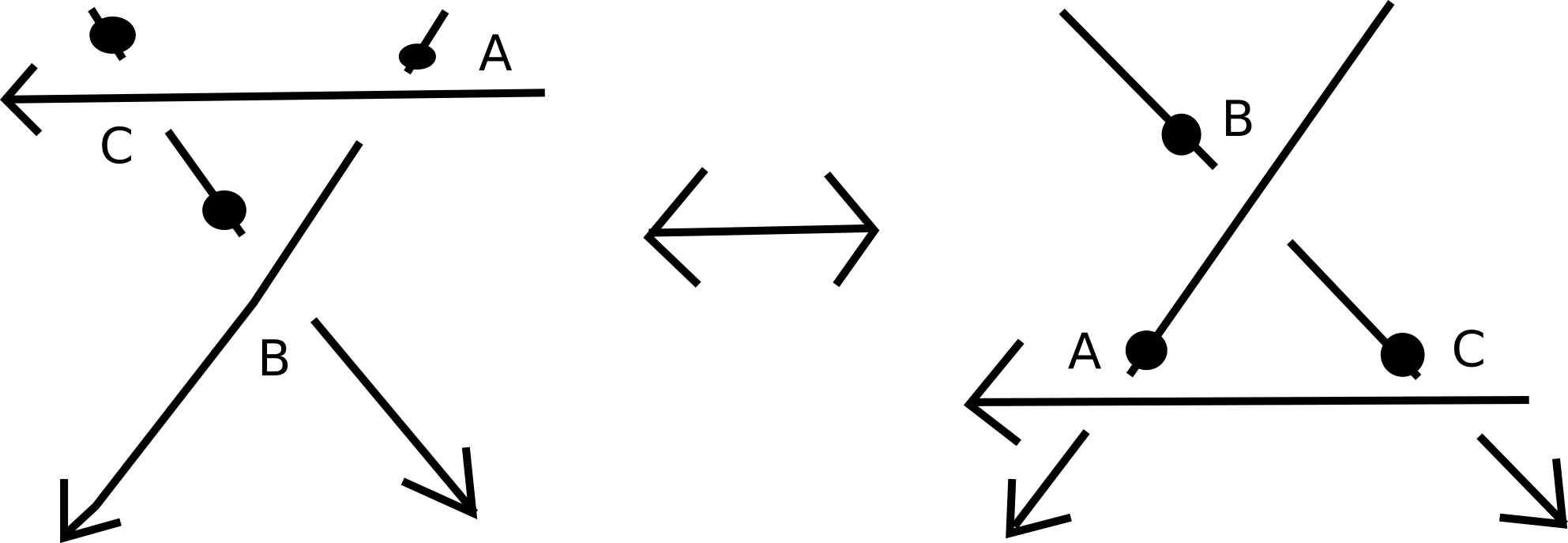}
  \caption{The third Reidemeister move.}
  \label{reidemeisterthree}
\end{figure}

\begin{figure}[!h]
\begin{subfigure}{.5\textwidth}
  \centering
  \includegraphics[width=.8\linewidth]{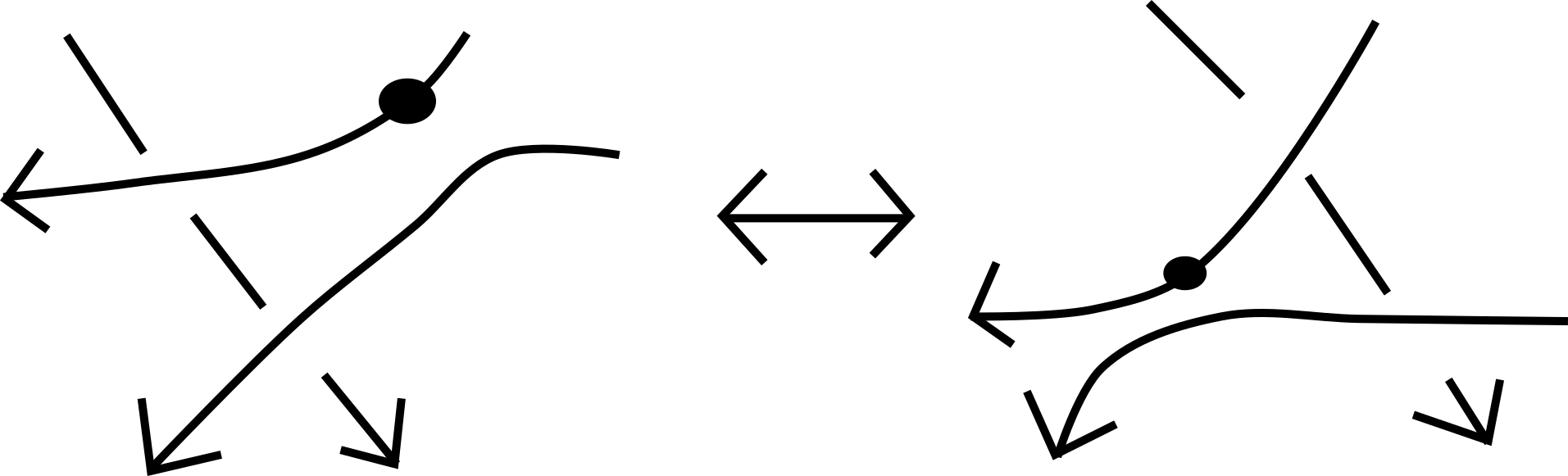}
  \caption{Smoothing crossing $A$.}
  \label{reidemeisterthreea}
\end{subfigure}
\begin{subfigure}{.5\textwidth}
  \centering
  \includegraphics[width=.8\linewidth]{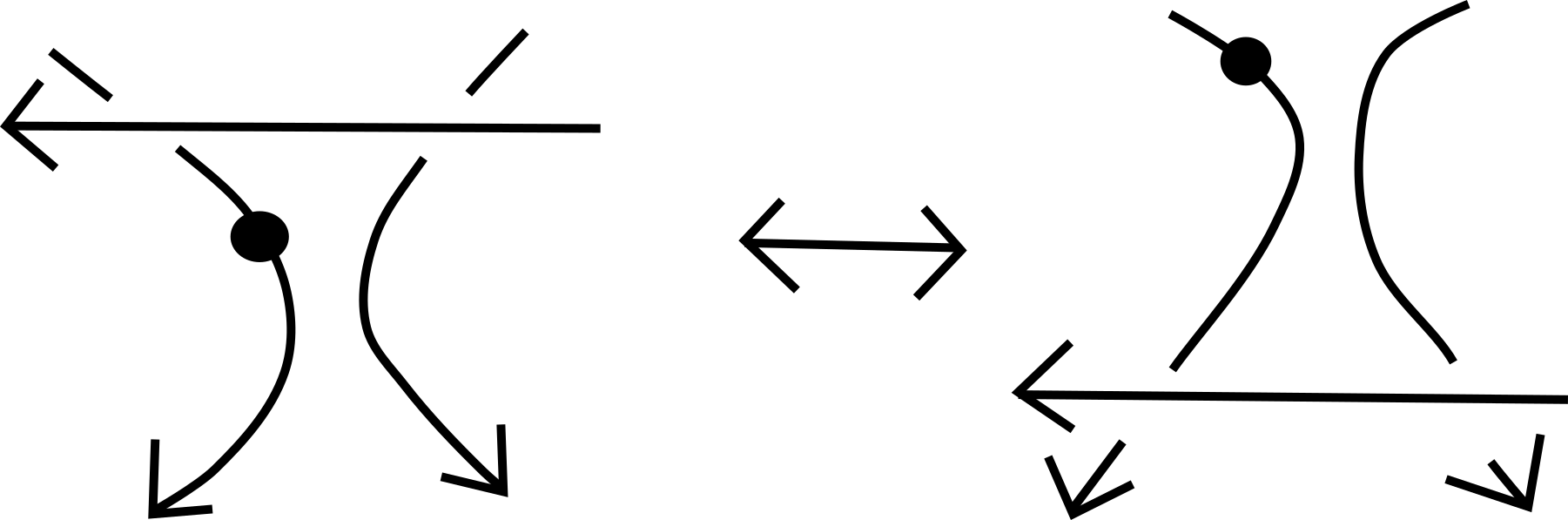}
  \caption{Smoothing crossing $B$.}
    \label{reidemeisterthreeb}
\end{subfigure}%
\caption{Smoothing crossings in Fig. \ref{reidemeisterthree}.}
\end{figure}

\begin{figure}[!h]
\centering
  \includegraphics[scale=.12]{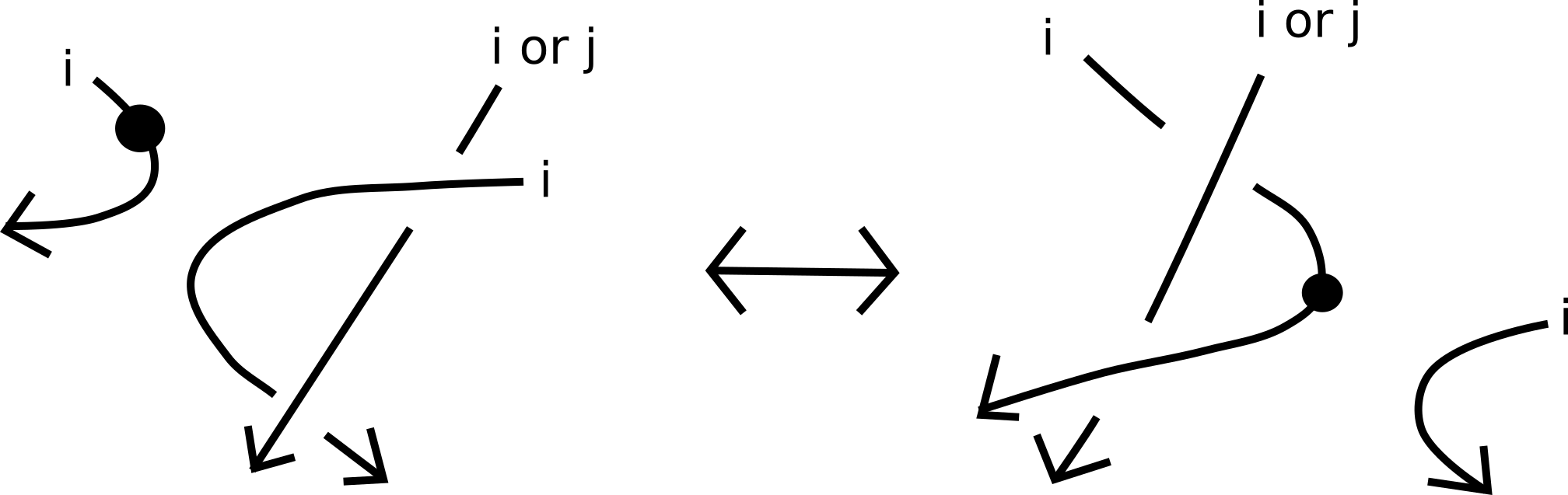}
  \caption{Smoothing crossing $C$ in Fig. \ref{reidemeisterthree}.}
    \label{reidemeisterthreec}
\end{figure}

The case of crossing C, pictured in Fig. \ref{reidemeisterthreec}, deserves a few more words: because $c$ is a self-crossing, we assume that the two strands that cross there belong to the same component, say $i$. If the third strand belongs to a different component $j$, we can ignore it in computing the invariant, and the resulting links are clearly isotopic, so the wriggle polynomial on either side of the move is the same. In the case that the third strand also belongs to component $i$, the other two crossings are self-crossings, so we must look at their contribution to the wriggle number, and that contribution will depend on whether after the smoothing the third strand is part of component one or component two. Let's suppose it belongs to component one: the picture on the left of Fig. \ref{reidemeisterthreec} then contains two positive crossings, one going over component two and one going under component two, so its wriggle number contribution is $(+1)-(+1)=0$, while the picture on the right contains two self-crossings of component one of the link, which are not counted in the wriggle number, so the two crossings also contribute $0$ to the wriggle number. All the other crossings outside the small region contribute the same to the wriggle number before and after the move (since the same component is chosen as component one), so both sides of Fig. \ref{reidemeisterthreec} have the same wriggle number.

Overall, we proved that the corresponding crossings on both sides of Fig. \ref{reidemeisterthree} have the same wriggle number, so the self-crossing wriggle polynomial is left unchanged by Reidemeister move three.

Finally, the case of the mixed Reidemeister move is pictured in Fig. \ref{mixedreidemeister}. For both versions of the move, the resulting links are isotopic and the same component is picked as component one, so the polynomial is invariant under this move as well. This completes the proof that the self-crossing wriggle polynomial is an invariant of virtual tangles.

\begin{figure}[!h]
  \centering
  \includegraphics[scale=.12]{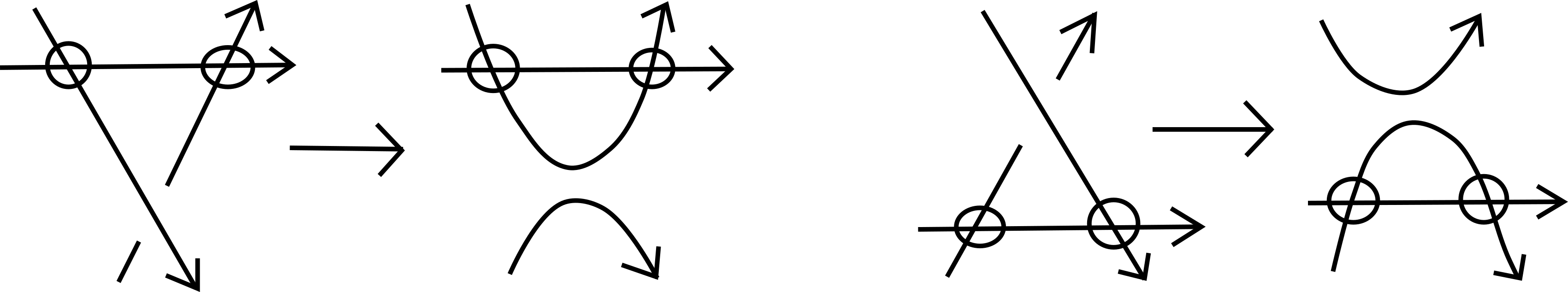}
  \caption{The mixed Reidemeister move case. The two links are clearly isotopic.}
  \label{mixedreidemeister}
\end{figure}

\end{proof}

The next step of the proof requires showing that the extension of the invariant to virtual tangles with double points always vanishes on tangles with two double points.

\begin{proof}
Let's consider a tangle $T$ with two double points $d, d'$, and resolve said double points according to the rule of section \ref{vassilievinvariants}. We will get the following expression (where, for notational convenience, we replaced the variables with the double points and their resolution):
\begin{equation}\label{resolutionequation}W_{sc, T}(d, d')=W_{sc, T}(++)-W_{sc,T}({+-})-W_{sc, T}({-+})+W_{sc, T}({--}).\end{equation}  A schematic version of the resulting expression is shown in Fig. \ref{proofresolution}, where the top crossing is $d$ and the bottom crossing is $d'$.

\begin{figure}[!h]
  \centering
  \includegraphics[scale=.12]{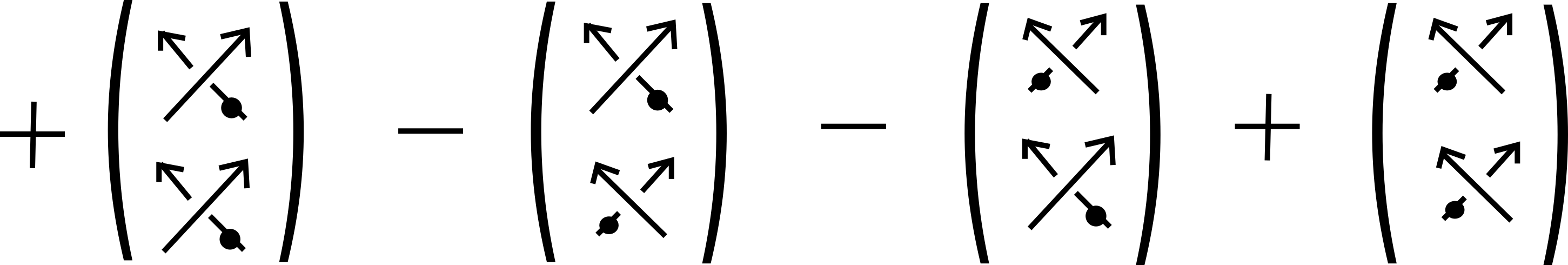}
  \caption{A schematic representation of equation \ref{resolutionequation} The top crossing represents $d$ and the bottom crossing $d'$.}
  \label{proofresolution}
\end{figure}

Note that a priori each of $d,d'$ could either be a self-crossing of a given component, or a crossing involving multiple components, and we will need to account for any possible combination.
Let's start by looking at the terms we get from $d$ and $d'$ themselves. Clearly, if either of the double points is at a crossing of two components, its resolutions are not self-crossings and as such do not contribute to the invariant. Now let's look at the case where $d$ is a self-crossing of component $i$ and $d'$ is a self-crossing of component $j$. If we smooth $d$ in the first two terms of the expansion we get isotopic links, and the same thing is true when we smooth the last two terms of the expansion; since these terms come with alternating signs, the total contribution of $d$ is zero, and an analogous argument shows that the total contribution of $d'$ is also zero. 

In the case where both $d$ and $d'$ are self-crossings of component $i$, we need to be a little more careful. Once again, let's focus on the first two terms of the expansion: it turns out that the contribution of $d'$ to the wriggle number is the same in both terms. Either $d'$ is a self-crossing of the link after the smoothing, and it doesn't contribute to the wriggle number; or both the sign of $d'$ AND which strand goes over change from one term to another, so it contributes the same amount ($\pm1$) in both terms. Every other crossing that is not $d, d'$ clearly contributes the same amount to both terms, so the wriggle number of the first two terms is the same; because of the alternating sign in front, the contribution of $d'$ in the first two terms cancels out, and for similar reasons the contribution of the last two terms also cancels out.
We can also use a similar argument with the third and fourth term of the expansion, as well as in the case of smoothing $d'$, pairing the first and third term and the second and fourth term. As a result, the net contribution of the crossings $d$ and $d'$ to the invariant is zero.

Now we need to look at how resolving $d$ and $d'$ affects the contribution of any \emph{other} crossing. Once again, the idea will be to pair up the terms.
\begin{itemize}
\item if $d, d'$ are not self-crossings, they will not contribute to any wriggle number, so each term in the expansion has the same invariant. Because of the alternating signs, the total contribution of every other crossing is zero.
\item if $d$ is a self-crossing of component $i$ and $d'$ is a self-crossing of component $j$, $d$ might affect the wriggle number from smoothing any self-crossing $c$ of component $i$, and similarly $d'$ might affect the wriggle number from smoothing self-crossings $c'$ of component $j$. Looking at $d$ and $c$ (since the $d', c'$ case is analogous) shows that the contribution of $d$ to $L_c$ is always the same (either zero if $d$ is a self-crossing of $L_c$, or always $\pm1$ depending which strand in $d$ belongs to component one of $L_c$). Since all terms have the same wriggle number, by the alternating signs the total contribution of $c$ is zero. This argument applies to every crossing that is not $d$ or $d'$, so the total contribution of every other crossing is zero.
\item Finally, in the case where $d$ and $d'$ both belong to component $i$, we need to make sure that the contribution of any self-crossing of component $i$ ultimately cancels out in the expansion. If in $L_c$ either $d$ or $d'$ is a self-crossing the argument of the previous case applies; otherwise a straightforward check of every possible choice of which strand in $d$ and $d'$ belongs to component one of $L_c$ shows that the wriggle number of $L_c$ is the same in every term, so once again the total contribution of every other crossing is zero because of the alternating sign.
\end{itemize}

In all of the above cases, the total contribution of any crossing is zero; this means $W_{sc, T}(d, d')=0$, completing this part of the proof.
\end{proof}

Finally, to show that the invariant doesn't vanish on a tangle with one double point we will consider a one-component tangle, i.e. a long virtual knot. It is easy to see that the extension of $W_{sc, T}$ to the tangle of Fig. \ref{pscnonzero} is $W_{sc, T}(t_1)=t_1^2-1$. This completes the proof of the theorem.

\begin{figure}[!h]
\centering
\includegraphics[scale=.15]{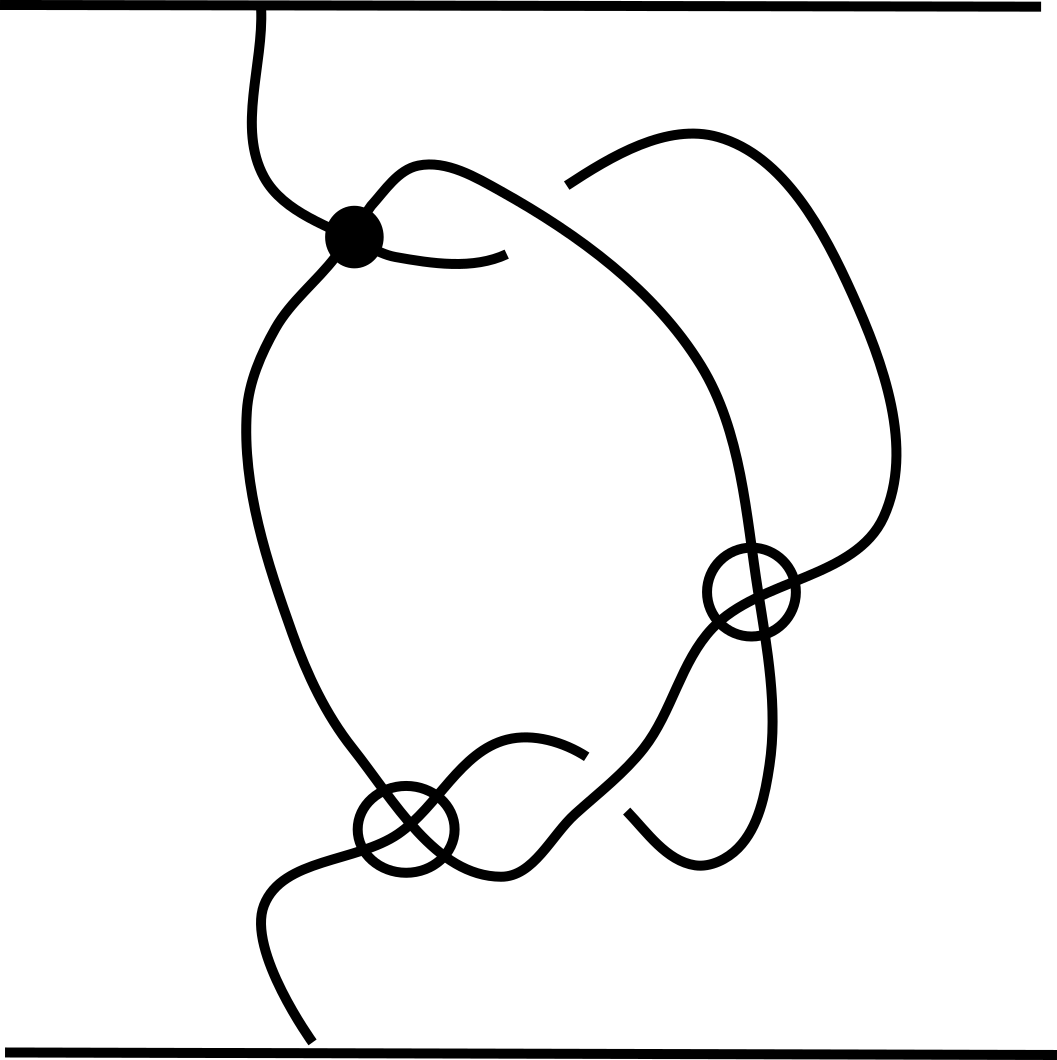}
\caption{A singular virtual tangle (in fact, a singular long virtual knot) for which $W_{sc, T}$ is nonzero.}
\label{pscnonzero}
\end{figure}

\begin{remark}
It is easy to see that if the tangle has only one component, this invariant reduces to the wriggle polynomial invariant of \cite{linkingnumberaffineindex}. 
As a consequence of Theorem \ref{wrigglepolyvirtualtanglesinvariant}, the wriggle polynomial for virtual knots is also an order one Vassiliev invariant of virtual knots. 
The proof that it vanishes on any knot with two double points is identical to the one for the tangle version, and the long knot in Fig. \ref{pscnonzero} is a long virtual knot with one double point whose wriggle polynomial is nonzero.
\end{remark}

\subsection{Properties of the invariant}
\label{strengthofinvariants}

In this section we will study a couple of properties of the self-crossing wriggle polynomial.

\begin{prop}
Let $T$ be a virtual tangle on $n$ components, and $T'$ a virtual tangle obtained from $T$ by reversing the orientation of component $T_i$. Then $$W_{sc, T'}(t_1,\ldots, t_i, \ldots, t_n)=W_{sc, T}(t_1, \ldots, t_i^{-1}, \ldots, t_n)$$
\end{prop}

\begin{proof}
Consider the effect of swapping the orientation of a component on a self-crossing of said component, pictured in Fig. \ref{swappingorientation}. Changing the orientation of the component will change the direction of both strands of a self-crossing, thus preserving their sign. \begin{figure}[!h]
\centering
\includegraphics[scale=.15]{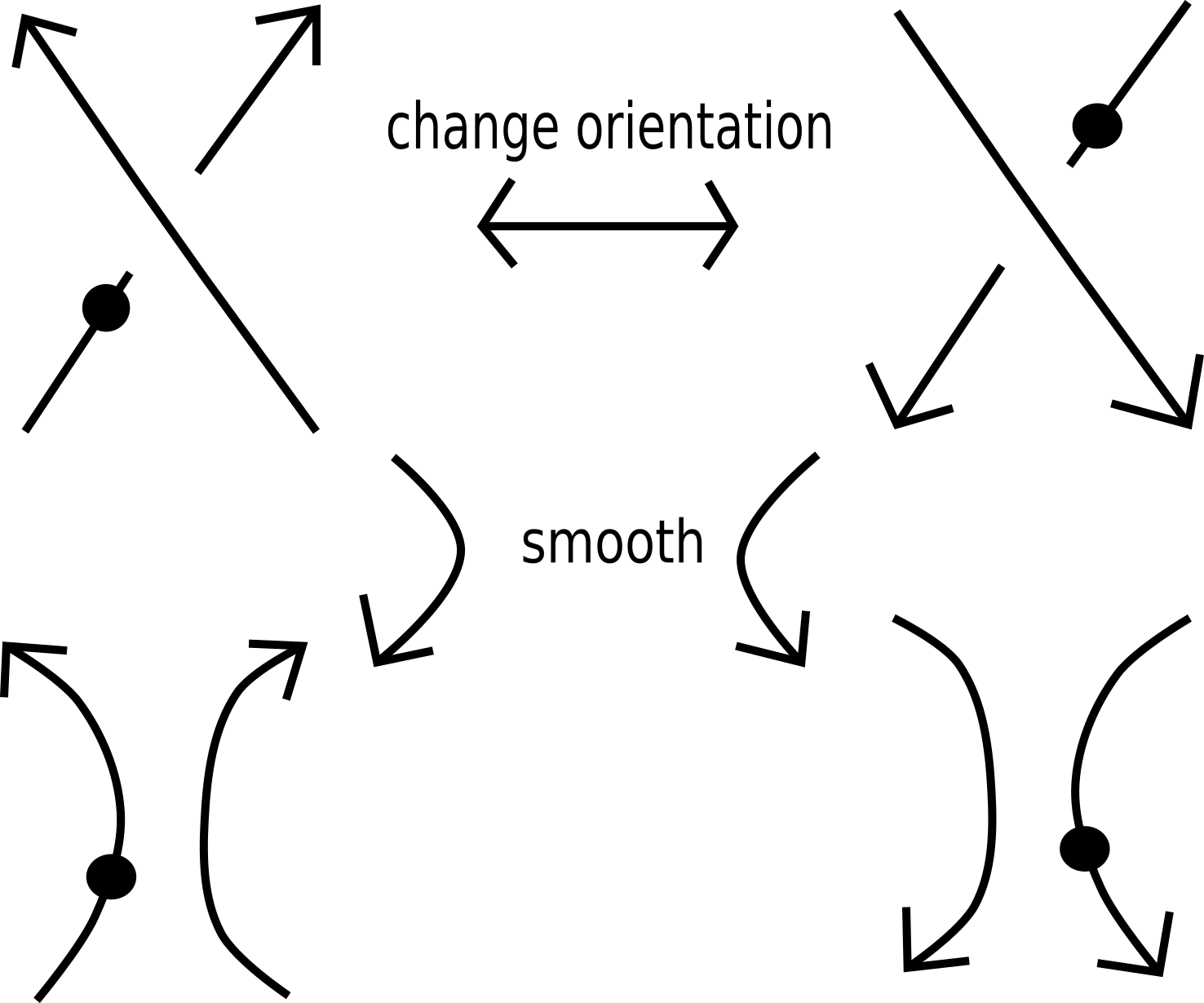}
\caption{The result of swapping the orientation of a component on the self-crossing wriggle polynomial.}
\label{swappingorientation}
\end{figure}
After smoothing a self-crossing, the resulting links are isotopic as unoriented links. 
Since both strands of the link swap orientation, the sign of every other crossing is also preserved; so the only difference in computing the two wriggle numbers comes from the choice of component one.
As mentioned in section \ref{virtualknotsandtangles}, exchanging components one and two will change the sign of the wriggle number. So $t_i^{W(L_c)}$ turns into $t_i^{-W(L_c)}$ for each self-crossing of $T_i$, which means we're replacing $t_i$ with $t_i^{-1}$.

\end{proof}

\begin{prop}
\label{additivity}
Let $T$ and $U$ be two oriented virtual tangles such that $T\#U$ is well-defined, neither of them containing a long component whose two ends belong to the same boundary (i.e. no ``cups'' or ``caps'' allowed). Assume that in the connected sum component $T_i$ glues to component $U_{\sigma(i)}$ for all $i$ for the appropriate permutation $\sigma$. Then $$W_{sc,T\#U}=W_{sc,T}+W_{sc,U}\in \Z[t_1, \ldots, t_n, u_1, \ldots, u_m]/R,$$
where $R$ is the set of relations $\{t_i=u_{\sigma(i)}\}$
\end{prop}

\begin{cor}\label{vsladditivity}
The self-crossing wriggle polynomial is additive and commutative under connected sum of virtual string links: if $T, U$ are virtual string links on $n$ strands, then $W_{sc, T\#U}=W_{sc, T}+W_{sc, U}=W_{sc, U\#T}$.
\end{cor}

\begin{proof}
Suppose we're dealing with two tangles $T, U$ with no closed component; each tangle has its own self-crossing polynomial $p_{sc}(T)\in\Z[t_1, \ldots, t_n]$, $p_{sc}(U)\in\Z[u_1, \ldots, u_m]$ (we use different variables to keep track of the components more easily).
Assuming that the connected sum $T\#U$ is well-defined, each long component of $T$ must be connected to a long component of $U$; suppose $T_i$ is glued to $U_{\sigma(i)}$ for an appropriate permutation $\sigma$ to form the $i$-th component of $T\#U$. Let $R$ be the set of relations of the form $t_i=u_{\sigma(i)}$ that come from the above identification.
Note that any closed components of $T$ or $U$ do not glue to anything, so their contribution to the self-crossing wriggle polynomial is unchanged by the glueing.
Let us now consider the long components: on the one hand the $i$-th long component of $T\#U$ contributes to the self-crossing polynomial the terms
$$\sum_{c\in (T\#U)_i}sgn(c)(t_i^{W(L_C)}-1).$$

But the self-crossings of $(T\#U)_i$ are the union of the self-crossings of $T_i$ with those of $U_{\sigma(i)}$. We can then rewrite the above expression as

\begin{equation}\begin{split}\label{tiusigmai}\sum_{c\in (T\#U)_i}sgn(c)(t_i^{W(L_c)}-1)=\sum_{\substack{c\in (T\#U)_i\\c\in T_i}}&sgn(c)(t_i^{W(L_c)}-1) +\\&+\sum_{\substack{c\in (T\#U)_i \\ c\in U_{\sigma(i)}}}sgn(c)(t_i^{W(L_c)}-1).\end{split}\end{equation}
On the other hand, the sum of the contributions of the self-crossings of $T_i$ and the self-crossings of $U_{\sigma(i)}$ is
$$\sum_{c\in T_i}sgn(c)(t_i^{W(L_c)}-1)+\sum_{d\in U_{\sigma(i)}}sgn(d)(u_{\sigma(i)}^{W(L_d)}-1).$$
The key observation is that the glueing has no effect on the wriggle number. Because the knottedness of each tangle is contained in its own square, glueing $U_{\sigma)i)}$ does not add any new crossings to any of the links obtained by smoothing a self-crossing of $T_i$, and the same thing is true reversing the roles of $T_i$ and $U_{\sigma(i)}$. 
So up to identifying the dummy variables $t_i$ and $u_{\sigma(i)}$ (which we do by quotienting out by $R$), the terms in equation \ref{tiusigmai} are the same as the contribution of $(T\#U)_i$; if we took the sum over all components, we would thus get
$$W_{sc,T\#U}=W_{sc,T}+W_{sc,U}\in \Z[t_1, \ldots, t_n, u_1, \ldots, u_m]/R.$$

The reason we need to avoid ``cups'' and ``caps'' in our glueing is that they identify two distinct components of the tangle the get glued to. This potentially changes some mixed crossings into self-crossings, and can affect the value of the polynomial in unexpected ways.
Also note that $T\#U$ and $U\#T$ are not necessarily both defined, and even when they are there is no guarantee that the components glue in the same way. This means that, generally speaking, $p(T\#U)\neq p(U\#T)$.

\end{proof}

\begin{proof}[Proof of Corollary \ref{vsladditivity}]
In a virtual string link the endpoints don't permute; this means that $\sigma=id$ in Proposition \ref{additivity}, which shows that the invariant is simply additive. Moreover, the self-crossings of component $(T\#U)_i$ are the same as those of component $(U\#T)_i$, so we also get commutativity under connected sum.
\end{proof}

\section{Future work}
The invariant we defined in section \ref{wrigglepolynomial} has one major flaw: it does not detect the way that the components of $T$ are linked together, and it is really just the sum of the wriggle polynomials of each separate component. We are currently investigating ways of strengthening the invariant to detect the linking between components in a way that generalizes the self-crossing wriggle polynomial. We are at the same also pursuing a slightly different approach, derived from the following fact: in the virtual knot case, the Wriggle polynomial coincides with the Affine Index polynomial, as shown in \cite{linkingnumberaffineindex}. It is easy to see that this is also true for long virtual knots, since the point at infinity does not affect either the labeling or the smoothing of the crossings. We can then construct a ``self-crossing'' affine index polynomial in the obvious way, ensuring it coincides with the self-crossing wriggle polynomial, and take that as a starting point for our generalization.

We have completed preliminary work on defining an affine index polynomial for virtual tangles, in a way that extends the already existing generalization for compatible virtual links found in \cite{virtualknotcobordismaffineindex}, by using a bi-label for every arc of the tangle, separately keeping track of the self-crossings and the mixed crossings. We are currently studying some of the properties of the new invariant. We are currently hard at work in studying the properties of said invariant, other possible ways of generalizing it, and if it is possible to convert these results to the wriggle polynomial setting.

The author would like to acknowledge his institution, Oxford College of Emory University, for research support during the development of this paper.

\bibliographystyle{amsalpha}
\bibliography{fullbibliography}

\end{document}